\documentclass[12pt,letterpaper]{article}

\usepackage{amssymb,amsfonts,amscd,amsthm}
\usepackage[all,arc]{xy}
\usepackage{enumerate, bbm}
\usepackage{mathrsfs}
\usepackage{tikz}
\usepackage[left=0.9in,top=0.9in,right=0.9in,bottom=0.9in]{geometry}
\usepackage{mathtools}
\usepackage{hyperref}
\usepackage{color}
\usepackage{graphicx}
\usepackage{enumitem} 
\graphicspath{ {TexImages/} }

\newtheorem{thm}{Theorem}[section]

\newtheorem{prop}[thm]{Proposition}
\newtheorem{lem}[thm]{Lemma}
\newtheorem{claim}[thm]{Claim}

\newtheorem{quest}[thm]{Question}

\newtheorem{conj}[thm]{Conjecture}

\theoremstyle{definition}
\newtheorem{defn}{Definition}

\hypersetup{
	colorlinks,
	citecolor=blue,
	filecolor=blue,
	linkcolor=blue,
	urlcolor=blue,
	linktocpage
}

\setenumerate[1]{label=\thesection.\arabic*.} 
\setenumerate[2]{label*=\arabic*.} 
\setlist[enumerate]{itemsep=2ex, topsep=2ex} 
\setlist[itemize]{itemsep=2ex, topsep=2ex}

\newcommand{\E}{\mathbb{E}}

\newcommand{\Om}{\Omega}

\newcommand{\pa}{\partial}

\renewcommand{\l}{\left}
\renewcommand{\r}{\right}

\newcommand{\sm}{\setminus}
\newcommand{\sub}{\subseteq}
\newcommand{\Mt}{T}

\renewcommand{\c}[1]{\mathcal{#1}}

\newcommand{\f}[2]{\frac{#1}{#2}}
\newcommand{\floor}[1]{\l\lfloor #1\r\rfloor}
\newcommand{\ceil}[1]{\l\lceil #1\r\rceil}

\title{Maximal Independent Sets in Clique-free Graphs}
\author{Xiaoyu He \footnote{Dept.\ of Mathematics, Stanford University {\tt alkjash@stanford.edu}. This material is based upon work supported by the National Science Foundation Graduate Research Fellowship under Grant No. DGE-1656518.}\and Jiaxi Nie \footnote{Dept.\ of Mathematics, UCSD {\tt jin019@ucsd.edu}.} \and Sam Spiro\footnote{Dept.\ of Mathematics, UCSD {\tt sspiro@ucsd.edu}. This material is based upon work supported by the National Science Foundation Graduate Research Fellowship under Grant No. DGE-1650112.}}
\date{\today}
\begin{document}
	\maketitle
	\begin{abstract}
		Nielsen proved that the maximum number of maximal independent sets (MIS's) of size $k$ in an $n$-vertex graph is asymptotic to $(n/k)^k$, with the extremal construction a disjoint union of $k$ cliques with sizes as close to $n/k$ as possible.  In this paper we study how many MIS's of size $k$ an $n$-vertex graph $G$ can have if $G$ does not contain a clique $K_t$.  We prove for all fixed $k$ and $t$ that there exist such graphs with $n^{\lfloor\frac{(t-2)k}{t-1}\rfloor-o(1)}$ MIS's of size $k$ by utilizing recent work of Gowers and B. Janzer on a generalization of the Ruzsa-Szemer\'edi problem. We prove that this bound is essentially best possible for triangle-free graphs when $k\le 4$.
	\end{abstract}

	\section{Introduction}

	If $G$ is a graph, a set $I\sub V(G)$ is said to be a \textit{maximal independent set} or simply an MIS if it is an independent set which is maximal with respect to set inclusion. Efficient algorithms for finding maximal independent sets (or equivalently, maximal cliques in the complement) have many applications
	. For example, clique  finding algorithms have been used in social network analysis to find closely-interacting community of agents in a social network~\cite{harary1957}. Other applications include bioinformatics~\cite{koch2001,SAMUDRALA1998}, information retrieval~\cite{augustson1970}, and computer vision~\cite{horaud1989}. Because of this, considerable attention has been given to the problem of efficiently generating all MIS's of a given graph, see for example~\cite{eppstein2010,johnson1988, lawler1980, makino2004, paull1959, tsukiyama1977}.

	In this paper we study extremal problems involving the maximum number of MIS's in a graph with given parameters.  The main result in this area is the following, which was proven independently by Miller and Muler~\cite{miller1960problem} and by Moon and Moser~\cite{moon1965cliques}. 
	\begin{thm}[\cite{miller1960problem,moon1965cliques}]\label{thm:MoonMoser}
		Let $m(n)$ denote the maximum number of MIS's in an $n$-vertex graph.  If $n\ge 2$, then
		\[m(n)=\begin{cases}
			3^{n/3} & n\equiv 0\mod 3,\\ 
			4\cdot 3^{(n-4)/3} & n\equiv 1\mod 3,\\ 
			2\cdot 3^{(n-2)/3} & n\equiv 2\mod 3.
		\end{cases}\]
	\end{thm}
	The lower bound when $3|n$ is achieved by a disjoint union of triangles.  It was shown by Hujter and Tuza~\cite{hujtera1993number} that the number of MIS's drops considerably for triangle-free graphs.
	
	\begin{thm}[\cite{hujtera1993number}]\label{thm:HujterTuza}
		Let $m_3(n)$ denote the maximum number of MIS's in an $n$-vertex triangle-free graph.  If $n\ge 4$, then
		\[m_3(n)=\begin{cases}
			2^{n/2} & n\equiv 0\mod 2,\\ 
			5\cdot 2^{(n-5)/2} & n\equiv 1\mod 2.
		\end{cases}\]
	\end{thm}
	We note that Theorem~\ref{thm:HujterTuza} has found a number of applications in recent years.  In particular, it was utilized by Balogh and Pet\v{r}\'i\v{c}kov\'a~\cite{balogh2014number} to count the number of maximal triangle-free graphs on $n$ vertices, and a variant of Theorem~\ref{thm:HujterTuza} was used by Balogh, Liu,  Sharifzadeh, and Treglow~\cite{balogh2018sharp} to count the number of maximal sum-free subsets of the first $n$ integers.

	Many other variants of Theorem~\ref{thm:MoonMoser} have been studied, with this including work of Wilf~\cite{wilf1986} who determined the maximum number of MIS's in trees, and by F\"uredi~\cite{furedi1987} and Griggs, Grinstead and Guichard~\cite{griggs1988} who determined the maximum number of MIS's in connected graphs.  Of particular interest to us is the following result of Nielsen~\cite{nielsen2002number}, which considers maximizing the number of MIS's of a given size.
	
	\begin{thm}[\cite{nielsen2002number}]\label{thm:Nielsen}
		Let $m(n,k)$ denote the maximum number of MIS's of size $k$ that an $n$-vertex graph can have.  If $s\in \{0,1,\ldots,k-1\}$ with $n\equiv s\mod k$, then
		\[m(n,k)=\floor{n/k}^{k-s}\ceil{n/k}^{s}.\]
	\end{thm}
	The lower bound of this theorem comes from the disjoint union of $k$ cliques of sizes as close to $n/k$ as possible.  
	
	In this paper, we study the ``intersection'' of the problems of Hujter and Tuza~\cite{hujtera1993number} and of Nielsen~\cite{nielsen2002number}. 
	To this end, we define $m_t(n,k)$ to be the maximum number of MIS's of size $k$ in a $K_t$-free graph on $n$ vertices.  It seems that understanding $m_t(n,k)$ in general is significantly more difficult than the problems we have discussed up to this point.  In fact, we do not even know the order of magnitude of this function when $k$ and $t$ are fixed.  The best general lower bounds we can give are the following.
	
	\begin{thm}\label{thm:constructions}	Let $k\ge 1,\ t\ge 3$ be fixed and let $n\ge k$.
	\begin{enumerate}

	    \item[(a)] If $t=3$ and $k<4$, then
	           \[m_3(n,k)=\Om(n^{\floor{\f{k}{2}}}),\]
	    		and if $t\ge 4$ and $k<2(t-1)$, then
	    		\[m_t(n,k)\ge n^{\floor{\f{(t-2)k}{t-1}}-o(1)}.\]
	    		
	    \item[(b)] If $t=3$ and $k\ge4$, then
            \[m_3(n,k)=\Om(n^{\f{k}{2}}),\]
            and if $t\ge 4$ and $k\ge2(t-1)$, then
	    	\[m_t(n,k)\ge n^{\f{(t-2)k}{t-1}-o(1)}.\]
	\end{enumerate}
	\end{thm}
	
	We suspect that the bounds of Theorem~\ref{thm:constructions} are essentially tight, see Conjecture~\ref{conj:main}. 
	
	Our proof of Theorem~\ref{thm:constructions} hinges on a recent construction of Gowers and B. Janzer~\cite{gowers2020generalizations} which generalizes a celebrated result of Ruzsa and Szemer\'edi~\cite{ruzsa1978}. Ruzsa and Szemer\'edi used Behrend's construction of large arithmetic progression-free sets to construct a graph on $n$ vertices and $n^2 e^{-O(\sqrt{\log n})}$ edges where each edge lies in at most one triangle. 
	Gowers and B. Janzer generalized this to show that for any $1\le r < s$, there is a graph on $n$ vertices with $n^r e^{-O(\sqrt{\log n})}$ copies of $K_s$ such that every $K_r$ lies in at most one $K_s$. This construction is one of the main building blocks in our proof of Theorem~\ref{thm:constructions}. We obtain the same type of error term, so the $n^{-o(1)}$ errors can be improved to $e^{-O(\sqrt{\log n})}$ in the statement of Theorem~\ref{thm:constructions}.

	We now turn our attention to upper bounds on $m_t(n,k)$. It is straightforward to prove that our lower bound from Theorem~\ref{thm:constructions}  when $k<t$ is tight up to the $o(1)$ term.
	\begin{prop}\label{prop:smallK}
		For any fixed $k<t$, we have \[m_t(n,k)=O(n^{\floor{\f{(t-2)k}{t-1}}})=O(n^{k-1}).\]
	\end{prop}
	
	Proving upper bounds on $m_t(n,k)$ in general seems like a very difficult problem.  However, when $t=3$ we are able to determine the order of $m_3(n,k)$ for $k\le 4$, and even an exact result for $m_3(n,2)$.  To state our result, we define a {\it comatching of order $n$} to be a complete bipartite graph $K_{\floor{n/2},\ceil{n/2}}$ minus a matching of size $\floor{n/2}$.
    \begin{thm}\label{thm:triangle-free}
		For $n\ge 8$ we have
		\[m_3(n,2)=\floor{n/2},\]
		and the unique graph achieving this bound is a comatching of order $n$.  Moreover, we have
		\begin{align*}m_3(n,3)&=\Theta(n),\\ m_3(n,4)&=\Theta(n^2).\end{align*}
	\end{thm}

	The rest of this paper is organized as follows. In Section~\ref{sec:constructions}, we prove Theorem~\ref{thm:constructions} by constructing graphs which have many MIS's from certain blowup hypergraphs. In Section~\ref{sec:small} we prove Proposition~\ref{prop:smallK}, showing that Theorem~\ref{thm:constructions} is tight for $k<t$, and we also prove the first part of Theorem~\ref{thm:triangle-free}, determining the exact value of $m_3(n,2)$. In Section~\ref{sec:t3}, we focus on the case of triangle-free graphs and complete the proof of Theorem~\ref{thm:triangle-free} by determining the order of magnitude of $m_3(n,3)$ and $m_3(n,4)$.  In Section~\ref{sec:partial} we prove a partial result towards upper bounding $m_3(n,5)$, and in Section~\ref{sec:hypergraphs} we briefly consider an analogous problem for hypergraphs.  We end with some open problems and concluding remarks in Section~\ref{sec:concluding}.
	
	{\bf Notation.} We use the shorthand $k$-MIS to mean a maximal independent set of size $k$. We define a \textit{transversal MIS} in an $k$-partite graph to be a $k$-MIS containing one vertex from each part. Here, it is implied that the $k$-partite graph comes with the data of a $k$-partition $\Pi$ if more than one exists.
	
	\section{Constructions}\label{sec:constructions}
	It is not difficult to see that a comatching of size $n$ has $\floor{n/2}$ transversal 2-MIS's, essentially because almost every vertex is in a unique non-edge between the two parts.  A similar idea gives the following.
	\begin{lem}\label{lem:conGowers}
		Let $r \ge 2$ be fixed.  For all $n$, there exists a balanced $r$-partite graph $G$ on $rn$ vertices such that $G$ contains at least $n^{r-1-o(1)}$ transversal $r$-MIS's.  Moreover, when $r=2$ there exists such a graph with $\Om(n)$ transversal 2-MIS's.
	\end{lem}
	\begin{proof}
		The case $r=2$ follows by considering a comatching.  In recent work of Gowers and B. Janzer~\cite{gowers2020generalizations}, it is shown that there exists a balanced $r$-partite\footnote{Strictly speaking their theorem statement does not guarantee that the graph is $r$-partite (let alone balanced), but one can show that such a graph exists by taking a random balanced $r$-partition of the graph they construct.} graph $G'$ on $rn$ vertices which has $n^{r-1-o(1)}$ copies of $K_{r}$, and is such that every $K_{r-1}$ of $G'$ is contained in at most one $K_r$.   Let $G$ be the $r$-partite graph on the same vertex set as $G'$ where an edge is added between two vertices $u,v$ if and only if they are in different parts and if $uv\notin E(G')$ (that is, $G$ is the ``$r$-partite complement'' of $G'$).  For example, if $G'$ was a comatching, then $G$ would be a matching.
		
		We claim that any set $\{u_1,\ldots,u_r\}$ which is a vertex set of a $K_r$ in $G'$ is a transversal $r$-MIS of $G$.  Indeed, each $u_i$ vertex necessarily comes from a distinct part of $G'$ since $G'$ is $r$-partite, and this set is an independent set of $G$ by construction.  Assume for contradiction that this set is not maximal, i.e. that $\{u,u_1,\ldots,u_r\}$ is also independent for some new vertex $u$.  If $u$ is in the same part as $u_i$, then this means that the set $S=\{u_1,\ldots,u_r\}\sm \{u_i\}$ forms a $K_{r-1}$ and that $S\cup \{u_i\},S\cup \{u\}$ are distinct $K_r$'s in $G'$ containing $S$.  This contradicts how $G'$ was constructed, so we conclude that every such set is a transversal $r$-MIS.  The result follows since $G'$ contains $n^{r-1-o(1)}$ copies of $K_r$.
	\end{proof}
	
	With this we can now prove Theorem~\ref{thm:constructions}(a), which handles the case that $k<2(t-1)$.  Here and throughout this section we construct graphs $G$ with $\Theta_{k,t}(n)$ vertices for ease of presentation.  These can easily be turned into $n$-vertex constructions with asyptotically as many MIS's by either randomly deleting a constant fraction of the vertices, or by adding dummy vertices which are adjacent to all the vertices in a given part of $G$ (assuming $G$ is $k$-partite).
	
	\begin{proof}[Proof of Theorem~\ref{thm:constructions}(a)]

	    The $k=1$ case follows by considering a star, and the $t=3$ case for $k=2,3$ follows by considering a comatching of size $n$ and a comatching of size $n-1$ together with an isolated vertex, respectively. It remains to prove the case $k\ge 2$ and $t\ge 4$.  More generally, we will show for any fixed $k,t$ in this range that $m_t(n,k)\ge n^{\floor{\f{(t-2)k}{t-1}}-o(1)}$, which will in particular give the result when $k<2(t-1)$.
	    
	    For $r\ge 2$, we let $G_r$ denote the construction from Lemma~\ref{lem:conGowers}, and we let $G_1$ denote an isolated vertex.  Note that for all $r\ge 1$ the graph $G_r$ has $n^{r-1-o(1)}$ distinct $r$-MIS's. Write $k=(t-1)q+s$ with $q=\floor{k/(t-1)}$ and $0\le s<t-1$.  Let $G$ be the graph formed by taking $q$ disjoint copies of $G_{t-1}$ and one copy of $G_s$ if $s\ne 0$.  Observe that $G$ is $K_t$-free since it is the disjoint union of $(t-1)$-partite graphs and that it has $\Theta_{k,t}(n)$ vertices.  Observe that the number of $k$-MIS's in $G$ is at least \[(n^{t-2-o(1)})^q\cdot (n^{s-1-o(1)})=n^{(t-2)q+s-1-o(1)}.\]
		A simple computation shows that this gives the stated bounds, proving the result.
	\end{proof}
	
	As we will show, the construction above is best possible in terms of order of magnitude for $t=3$ and $k\le 4$, but for $k=5$ one can do better. We give a sketch of such a construction below.
	
	Let $v_1,\ldots,v_5$ be the five vertices of a 5-cycle $C_5$ whose edges are $e_i=\{v_i,v_{i+1}\}$ for $1\le i\le 5$. Let $V_{i}$ be the set of functions from $\{e_i,e_{i+1}\}$ to $[n^{1/2}]$, which can also be considered as the set of all possible $n^{1/2}$-colorings of $e_i$ and $e_{i+1}$. Let $G$ be the graph on $V_{1}\cup V_{2}\cup V_{3}\cup V_{4}\cup V_{5}$ where two vertices (i.e. functions) $f\in V_i$, $g\in V_j$ form an edge if there exists an edge $e\in C_5$ with $v_i,v_j\in e$ (i.e. if $j=i+1$ or $i=j+1$) and if $f(e)\not=g(e)$.  Observe that $G$ has $5n$ vertices and that it is a subgraph of a blowup of $C_5$, so in particular it is triangle-free.
	
	We claim that $G$ has many MIS's.  Indeed, given an $n^{1/2}$-coloring $F$ of the edges of $C_5$, let $f_{i,F}\in V_i$ be the function that agrees with $F$ on $e_i$ and $e_{i+1}$. It is not hard to check that $\{f_{1,F},f_{2,F},f_{3,F},f_{4,F},f_{5,F}\}$ is a 5-MIS. Since there are $n^{5/2}$ possible $n^{1/2}$-colorings of the edges of $C_5$, we know that $G$ admits at least $n^{5/2}$ MIS of size $5$. Because $G$ is triangle-free and has $5n$ vertices, we conclude that $m_3(n,5)\ge (n/5)^{5/2}$.

The graph $G$ we constructed above can be viewed as a certain blowup of a $5$-cycle.  We will generalize this blowup construction to hypergraphs by making use of fractional matchings.

\begin{defn}
Given a hypergraph $H=(V,E)$ and $x\in V(H)$,  we define \[E_x=\{e\in E:x\in e\}.\]  A {\em fractional matching} of $H$ is a function $M$ that maps $E(H)$ to non-negative real numbers such that for each vertex $x$, 
$$
\sum_{e\in E_x}M(e)\le1.
$$
We define the \textit{total weight} of a fractional matching $M$ to be 
$$|M|=\sum_{e\in E}M(e).$$
\end{defn}
With this we can define our main construction.  This construction utilizes an $r$-graph $H$ on $[k]$ together with graphs from Lemma~\ref{lem:conGowers} whose parts are indexed by $[r]$.  As much as possible we let $x,y$ denote elements of $[k]$ with $i,j$ denoting elements of $[r]$.
\begin{defn}\label{def:construction}
Given an $r$-graph $H=([k],E)$ and a fractional matching $M$ of $H$, we define the $(M,n)$-blowup graph of $H$, denoted by $G_n[H,M]$, as follows. For each $e\in E$, let $G_e$ be a graph as in Lemma~\ref{lem:conGowers} such that each of its parts $U_{e,1}$, $U_{e,2},\dots,U_{e,r}$ has $n^{M(e)}$ vertices. For each $x\in [k]$, define $V_x$ to be the set of functions $f$ from $E_x$ to $\bigcup_{e,i} U_{e,i}$ such that $f(e)\in U_{e,i}$ if and only if $x$ is the $i$th smallest vertex in the edge $e\subset [k]$. We make $f\in V_x$, $g\in V_y$ adjacent in $G_n[H,M]$ if and only if there is an edge $e\ni x,y$ in $H$ and $f(e)\sim g(e)$ in $G_e$. 
\end{defn}

For example, if $H=C_5$ and $M$ is identically equal to 1/2, then this recovers the example we detailed above for $C_5$, where in this example each graph $G_e$ was a comatching on two copies of $[n^{1/2}]$.

In general, observe that the number of vertices of $G_n[H,M]$ is $\sum_{x\in [k]} \prod_{e\in E_x}  n^{M(e)}\le kn$, where here we used that $M$ is a fractional matching.  Also observe that $G_n[H,M]$ contains edges between some $V_x,V_y$ only if $xy$ is an edge in the shadow graph $\pa H$ of $H$, which we recall is the graph on $V(H)$ which has the edge $xy\in \pa H$ if and only if there is a hyperedge of $H$ containing $x,y$.  In particular, if $\partial H$ is $K_t$-free, then so is $G_n[H,M]$.  As in the $C_5$ example, it turns out that $G_n[H,M]$ can contain many MIS's.

\begin{lem}\label{lem:blowupconstruction}
Given an $r$-graph $H$ on $[k]$ and a fractional matching $M$ of $H$, the graph $G_n[H,M]$ has at least $n^{|M|(r-1-o(1))}$ distinct $k$-MIS's.  Moreover, when $r=2$ it has at least $\Om(n^{|M|})$ distinct $k$-MIS's.
\end{lem}
As an example, when $G=C_{k}$ with $k$ even, there is a 1-parameter family of fractional matchings $M$ of weight $k$.  This lemma shows that each of these (very different) constructions gives $\Om(n^{k/2})$ distinct $k$-MIS's.
\begin{proof}
For each $e\in E$, let $\c{I}_e$ be the set of transversal $r$-MIS's of $G_e$. For any function $F$ which maps edges $e\in E(H)$ to sets in $\c{I}_e$, we define the set $I_F\subset V(H)$ to be the set of vertices $f$ of $G_n[H,M]$ (i.e. functions from some $E_x$ to $\bigcup_{e,i} U_{e,i}$) such that $f(e)\in F(e)$ for all $e$ in the domain of $f$. We claim that $I_F$ is a $k$-MIS of $G_n[H,M]$.

Observe that each $V_x$ has a unique element in $I_F$, namely the function $f_x$ which sends each $e\in E_x$ to the unique vertex in $F(e)\cap U_{e,i}$ whenever $x$ is the $i$th smallest vertex in $e$.  Also note that having $f_x\sim f_y$ in $H$ would imply that there is some hyperedge $e\ni x,y$ such that $f_x(e)\sim f_y(e)$ in $G_e$, a contradiction to $f_x(e),f_y(e)\in F(e)\in \c{I}_e$. Finally, assume there was some other $f\in V_x$ such that $\{f\}\cup I_F$ is still an independent set in $G_n[H,M]$. Because $f\not=f_x$, there is some $e\in E_x$ such that $f(e)\not=f_x(e)$. By construction $\{f_y(e):y\in e\}$ is an MIS of $G_e$, so this means $f(e)\sim f_y(e)$ in $G$ for some $y\not=x$, and hence $f\sim f_y$ in $G_n[H,M]$, a contradiction.

We conclude that each $I_F$ if an MIS of size $k$. The number of choices for $F$ is 
$$\prod_{e\in E(H)}|\c{I}_e|\ge\prod_{e\in E(H)}n^{M(e)(r-1-o(1))}=n^{|M|(r-1-o(1))}.$$
It is not hard to see that each $F$ gives a distinct independent set $I_F$, proving the result for $r>2$.  The case $r=2$ follows from the exact same reasoning except we now use the bound $|\c{I}_e|=\Om( n^{M(e)})$ guaranteed by Lemma~\ref{lem:conGowers}.
\end{proof}

In view of Lemma~\ref{lem:blowupconstruction} and the observation that $G_n[H,M]$ is $K_t$-free if $\partial H$ is $K_t$-free, we are left with the problem of finding $k$-vertex $r$-graphs with large fractional matchings and no $K_t$ in their shadow.  Tight cycles turn out to be one such example.

\begin{defn}
For integers $k\ge r\ge 2$, the $r$-uniform \textit{tight cycle} of length $k$, denoted by $TC^r_k$, is the $r$-graph with vertices $V=\{v_0, v_1, \dots, v_{k-1}\}$, whose edges are all $r$-sets in $V$ of the form $\{v_i,v_{i+1},\dots, v_{i+r-1}\}$, where all indices are modulo $k$.
\end{defn}
Observe that for $r=2$ this is simply the graph cycle $C_k$.

\begin{lem}\label{lem:cliquefree}
For any integer $r\ge2$ and  $k\ge 2r$, the shadow graph $\partial TC^r_k$ is $K_{r+1}$-free.
\end{lem}

\begin{proof}

Suppose for contradiction that $A$ is a set of vertices that induced a copy of $K_{r+1}$. Without lost of generality, we can assume $v_0\in A$. Then $A\backslash\{v_0\}$ must be contained in the set of neighbors of $v_0$. Partition the set of neighbors of $v_0$ into $r-1$ pairs $\{v_{i}, v_{i-r}\}$, $1\le i\le r-1$. Since $|A\backslash\{v_0\}|=r$, by the Pigeonhole Principle, there must be a pair $\{v_{i}, v_{i-r}\}\subset A$. Since $k\ge 2r$, it is not hard to check that $\{v_i, v_{i-r}\}=\{v_i,v_{i+(k-r)}\}$ is not contained in any edge of $TC_k^r$, and hence this is not an edge of $\partial TC^r_k$.  This gives a contradiction, completing the proof.
\end{proof}

\begin{proof}[Proof of Theorem~\ref{thm:constructions}(b)]
Let $H=TC^{t-1}_k$ with $k\ge 2(t-1)$, and let $M$ be the fractional matching with $M(e)=1/(t-1)$ for all $e$. By definition, $|M|=k/(t-1)$. Now consider the $(M,n)$-blowup graph $G_n[H,M]$. By Lemma~\ref{lem:blowupconstruction}, $G_n[H,M]$ has at least $n^{\frac{k(t-2-o(1))}{t-1}}$ distinct $k$-MIS's. Also, by Lemma~\ref{lem:cliquefree} we know that $\partial H$ is $K_t$-free, and hence $G_n[H,M]$ is $K_t$-free. Note that the number of vertices in $G_n[H,M]$ is $kn$, giving $m_t(n,k)\ge (n/k)^{\frac{k(t-2-o(1))}{t-1}}$ and hence proving the result for $k$ fixed.
\end{proof}
We note that any fractional matching $M$ of a $k$-vertex $r$-graph satisfies $|M|\le k/r$, and also that essentially any $(M,n)$-blowup of an $r$-graph with $r\ge t$ will contain copies of $K_t$.  Thus the lower bound of Theorem~\ref{thm:constructions}(b) is best possible using constructions of this form.

Lastly, we note that it is possible to generalize our construction to non-uniform hypergraphs $H$.  Indeed, the construction used in Theorem~\ref{thm:constructions}(a) can be viewed as a blowup of the non-uniform hypergraph consisting of $q$ disjoint edges of size $t-1$ and an edge of size $s$.  However, for $k\ge 2(t-1)$ the best constructions of this form will always come from a $(t-1)$-uniform hypergraph.

\section{The case $k<t$}\label{sec:small}
We start with a simple observation.
\begin{lem}\label{lem:t1}
    For all $t\ge 3$, we have
		\[m_t(n,1)=\begin{cases}
			n & n<t,\\ 
			t-2 & n\ge t.\\
		\end{cases}\]
\end{lem}
\begin{proof}
For $n<t$, the graph $G=K_n$ shows $m_t(n,1)\ge n$ which is best possible.  For $n\ge t$, we observe that any vertex $v\in V(G)$ which is an MIS must be adjacent to every other vertex of $G$.  If $|V(G)|\ge t$, then having $t-1$ such vertices would form a $K_t$.  Thus $m_t(n,1)\le t-2$ in this case, and taking $G$ to be a graph consisting of a $K_{t-2}$ which is adjacent to every other vertex shows that this is tight.
\end{proof}
With this we can prove Proposition~\ref{prop:smallK}.

\begin{proof}[Proof of Proposition~\ref{prop:smallK}]
We will prove this by showing
    \[m_t(n,k)\le t {n\choose k-1}.\]
Let $G$ be an $n$-vertex graph.  Observe that any $k$-MIS of $G$ can be formed by first choosing $k-1$ vertices $\{v_1,\ldots,v_{k-1}\}$ followed by a vertex $v_k$ which is a 1-MIS of the graph induced by $V(G)\sm \bigcup_i N(v_i)$.  There are at most ${n\choose k-1}$ ways to do this first step, and by Lemma~\ref{lem:t1} there are at most $t$ ways to do the second step, giving the stated bound.
\end{proof}

By using a more careful argument, we can obtain an exact formula for the maximum number of 2-MIS's in triangle-free graphs, which proves the first part of Theorem~\ref{thm:triangle-free}.
	\begin{prop}\label{prop:bipartite-case}
		For $n\ge 8$, we have \[m_3(n,2)=\floor{n/2}.\]  Moreover, the unique construction achieving this bound is the comatching of size $n$.
	\end{prop}
	\begin{proof}
		The lower bound follows by considering a comatching.  To complete the proof, let $G$ be an $n$-vertex triangle-free graph and let $\c{I}$ denote the set of 2-MIS's of $G$.  For the sake of contradiction, assume $|\c{I}|\ge \floor{n/2}$ and that $G$ is not isomorphic to a comatching.

		Let $\{x,y\}\in \c{I}$, and let $V_x,V_y,V_{x,y}$ be the set of vertices which are adjacent to $x$ but not $y$, adjacent to $y$ but not $x$, and adjacent to both $x$ and $y$, respectively.  Observe that $\{x,y\}$ being an MIS implies that these three sets partition $V(G)\sm \{x,y\}$.
		\begin{claim}\label{cl:I1}
			Let $\c{I}_1$ denote the set of 2-MIS's of $G$ which contain at least one of $x$ or $y$, and let $r=|\{|S|=1: S\in \{V_x,V_y\}\}|.$   Then\[|\c{I}_1|\le r+1.\]
		\end{claim}
		\begin{proof}
			Observe that any independent set $I$ containing $x$ must lie in $\{x,y\}\cup V_y$. Because $G$ is triangle-free, $V_y$ is an independent set, so the only MIS's containing $x$ are $\{x,y\}$ and $\{x\}\cup V_y$.  Similarly the only MIS's containing $y$ are $\{x,y\}$ and $\{y\}\cup V_x$.  From this the claim follows.
		\end{proof}
		\begin{claim}\label{cl:Vxy}
		    If $v\in V_{x,y}$, then $N(v)=\{x,y\}$.
		\end{claim}
		\begin{proof}
		    Any $u\in V(G)\sm \{x,y,v\}$ has a common neighbor with $v$, namely $x$ or $y$, so $G$ being triangle-free implies $u\notin N(v)$.
		\end{proof}
		\begin{claim}\label{cl:I2}
		    Let $\c{I}_2=\c{I}\sm \c{I}_1$, denote the set of 2-MIS's which do not contain either $x$ or $y$.  Every $I\in \c{I}_2$ contains $V_{x,y}$.
		\end{claim}
		\begin{proof}
		    By Claim~\ref{cl:Vxy}, if $I\sub V(G)\sm \{x,y\}$ is an independent set, then $I\cup V_{x,y}$ is also an independent set, so every MIS in $V(G)\sm \{x,y\}$ must contain $V_{x,y}$.
		\end{proof}
		We now break up the proof into several cases.  In our analysis we let $r$ be as in Claim~\ref{cl:I1}, and we will often make use of the following observation:
		\begin{equation}
		    |V_x|+|V_y|+|V_{x,y}|=n-2.\label{eq:n-2}
		\end{equation}

		\textit{Case $|V_{x,y}|\ge 2$.}  By Claim~\ref{cl:I2}, the only way $\c{I}_2\ne \emptyset$ is if $V_{x,y}$ is a 2-MIS.  In this case Claim~\ref{cl:Vxy} implies that $G=C_4$, contradicting the assumption $n\ge 8$.  Thus $\c{I}_2=\emptyset$, and by Claim~\ref{cl:I1} we have $|\c{I}|\le 3<\floor{n/2}$, a contradiction.

		\textit{Case $|V_{x,y}|=1$.}  Say $V_{x,y}=\{u\}$.  By Claim~\ref{cl:I2}, any 2-MIS of $\c{I}_2$ must use $u$ together with a vertex $v\in V_x\cup V_y$ which is adjacent to every other vertex of $V_x\cup V_y$.  The number of such $v$ is at most $r$, so by Claim~\ref{cl:I1} we find $|\c{I}|\le 1+2r$.  By \eqref{eq:n-2}, $r=2$ can only occur if $n=5$, so for $n\ge 8$ we again have $|\c{I}|\le 3<\floor{n/2}$, a contradiction.
		
		\textit{Case $V_{x,y}=\emptyset$.}  We claim that any $I\in \c{I}_2$ must use one vertex from $V_x$ and one vertex from $V_y$.  Indeed if we had $I\sub V_x$, then $I\cup \{y\}$ would be an independent set, contradicting the maximality of $I$.  
		
		With this claim we can write $\c{I}_2=\{\{u_1,v_1\},\ldots,\{u_m,v_m\}\}$ with $u_i\in V_x$ and $v_i\in V_y$.  Observe that $\{u_i,v_i\}\in \c{I}_2$ implies that, for example, $u_i$ is adjacent to every vertex of $\{x\}\cup V_y\sm \{v_i\}$, as otherwise $\{u_i,v_i\}$ would not be maximal.  In particular, $u_i\ne u_j$ and $v_i\ne v_j$ for all $i\ne j$, and hence $|\c{I}_2|\le \min\{|V_x|,|V_y|\}$.  Without loss of generality we can assume $|V_x|\le |V_y|$, so by Claim~\ref{cl:I1} we have
		
		\[|\c{I}|\le r+1+|V_x|.\]
		If $r=2$ then $n=4$ by \eqref{eq:n-2}, a contradiction.  Similarly $r=1$ implies $|V_x|=1$ and hence $|\c{I}|\le 3<\floor{n/2}$, a contradiction.  Thus we must have $r=0$, which by \eqref{eq:n-2} implies
		\[|\c{I}|\le 1+ |V_x|\le \floor{n/2}.\]
		The second equality can only hold if $|V_x|=\floor{n/2}-1$ and $|V_y|=\ceil{n/2}-1$.  The first equality can only hold if the only edges between $V_x$ and $V_y$ which are not present are $u_iv_i$ for all $i$.  In total this implies that $G$ is a comatching of size $n$, giving the desired contradiction.
	\end{proof}
    We note that a similar analysis can be used to determine $m_3(n,2)$ for all $n$, and in particular one can show 
    	\[m_3(n,2)=\begin{cases}
		2 & n=3,\\
		4 & n=4,\\ 
		5 & n=5,\\ 
		\floor{n/2} & n\ne 3,4,5.
	\end{cases}
	\]
	The comatching gives the lower bound for all $n\ne 5$, and the bound for $n=5$ is achieved by a $C_5$.  The comatching and $C_5$ are the unique constructions achieving these bounds, except when $n=6,7$, in which case there exist multiple constructions.  For example, if $G$ consists of a $C_4$ on $v_1v_2v_3v_4$ together with leaves $u_1,u_2$ adjacent to $v_1,v_3$, then $G$ is another 6-vertex triangle-free graph with 3 MIS's of size 2.

	\section{The cases $t=3$ and $k=3,4$}\label{sec:t3}
	In this section we study triangle-free graphs. The goal is to prove the second and third parts of Theorem~\ref{thm:triangle-free}, which state that $m_3(n,3) = \Theta(n)$ and $m_3(n,4) = \Theta(n^2)$, respectively. 
	
	We begin with a technical lemma, reducing the problem to counting only transversal MIS's in $k$-partite graphs. Given a $k$-partite graph $G$, We let $\Mt(G,k)$ denote the number of transversal $k$-MIS's of $G$. We note that the value of $\Mt(G,k)$ may depend on the implicit choice of the $k$-partition of $G$.
	
	\begin{lem}\label{lem:transversal}
	If $G$ is a triangle-free graph on $n$ vertices, then $G$ has an induced $k$-partite subgraph $G'\sub G$ satisfying
	\[
	\Mt(G',k) \ge (4k)^{-k} m(G,k).
	\]
	\end{lem}
	\begin{proof}
	    We first claim that if $G$ is triangle-free and contains at least one $k$-MIS, then the chromatic number of $G$ is at most $k+1$. Indeed, let $I = \{v_1,\ldots, v_k\}$ be any $k$-MIS of $G$. By maximality, every other vertex of $G$ must have a neighbor among the $v_i$. Let $U_i$ be the set of all $u\in V(G)\setminus I$ adjacent to $v_i$ but not to any $v_j$ for $j<i$. Since $G$ is triangle-free, each $U_i$ is an independent set, so $G$ has a partition into at most $k+1$ independent sets $U_1, U_2, \ldots, U_{k+1}$ (where $U_{k+1} = I$), proving the claim.
	
	    To reduce to the transversal case, for each composition $c=(c_1,\ldots, c_{k+1})$ of $k$ into $k+1$ nonnegative integers $c_1 + \cdots + c_{k+1} = k$, define $m_c$ to be the number of $k$-MIS's in $G$ with $c_i$ elements in $U_i$. As the total number of such compositions is $\binom{2k}{k} < 4^k$, there exists some $c$ for which $m_c \ge 4^{-k} m(G,k)$.
	    
	    Let $H$ be the induced subgraph of $G$ consisting only of those parts $U_i$ for which $c_i>0$. By definition, $H$ has $m_c$ $k$-MIS's which contain $c_i$ vertices from part $U_i$. At most $k$ values of $c_i$ are nonzero, so $H$ is $k$-partite. Define a random $k$-partition $\Pi$ on $H$ by partitioning each $U_i\subseteq V(H)$ uniformly at random into $c_i$ disjoint subsets. If $I$ is a $k$-MIS of $H$ with $c_i$ parts in $U_i$ for each $i$, the probability it is transversal with respect to $\Pi$ is
	    \[
	    \prod_{i=1}^{k+1} \frac{c_i!}{c_i^{c_i}} \ge k^{-k}.
	    \]
	    (Here $0^0 = 0! = 1$.)
	    It follows by linearity of expectation that $\E[T(H,k)] \ge k^{-k}m_c$, when $H$ is given the $k$-partition $\Pi$. Thus, there exists a $k$-partite triangle-free graph $G' = (H,\Pi)$ with $\Mt(G', k) \ge k^{-k}m_c \ge (4k)^{-k} m(G,k)$.
	\end{proof}
	
	Lemma~\ref{lem:transversal} shows that up to a multiplicative factor of $O_k(1)$, for the $t=3$ case it suffices to count only transversal MIS's in $k$-partite graphs. 
	
	\begin{lem}\label{lem:tripartite}
		If $G$ is a triangle-free tripartite graph, then \[\Mt(G,3) \le |V(G)|.\]
	\end{lem}

	\begin{proof}
	    Let $G$ be a triangle-free tripartite graph with parts $U$, $V$, and $W$.  We break down the proof into a series of structural claims about transversal MIS's in $G$.
		
		\begin{claim}\label{cl:linear-transversal}
			There is at most one transversal MIS containing any pair of vertices $u,v \in V(G)$.
		\end{claim}
		\begin{proof}
		    If $u$ and $v$ lie in the same part, then no transversal triples can contain both $u$ and $v$. Thus we may assume $u\in U$ and $v\in V$.  Suppose $(u,v,w),(u,v,w')\in U\times V\times W$ are transversal MIS's, so in particular they are independent sets.  We have $w\not\sim w'$ since they both lie in $W$, so $\{u,v,w,w'\}$ is an independent set.  Because $\{u,v,w\}$ is an MIS, this is only possible if $w'=w$, and hence there is only one such transversal MIS.
		\end{proof}
		
		If $S\subseteq V(G)$ we define $N_S(v)=N(v)\cap S$ and $N_S^c(v)=S\sm N(v)$.
		\begin{claim}\label{cl:d-exist}
			If $u\in U$ lies in at least two transversal MIS's, then there is an integer $d=d(u)$ such that for every transversal MIS $\{u,v,w\}$ containing $u$ with $v\in V$ and $w\in W$, we have $|N_U(v)|=d,\ |N_U(w)|=|U|-d-1$.
		\end{claim}
		\begin{proof}
		   Let $d=\max_{v\in N_V^c(u)} |N_U(v)|$ and let $v$ be a vertex achieving this maximum. Since $u$ lies in at least two transversal MIS's, Claim~\ref{cl:linear-transversal} implies the existence of an MIS $(u,v',w)\in U\times V\times W$ with $v'\ne v$. Because $v\in N_V^c(u)$ and $\{u,v',w\}$ is an MIS, we must have $v\in N(w)$.  Because $G$ is triangle-free, this means $N_U(v)\cap N_U(w)=\emptyset$, and because $N_U(v),N_U(w)\sub U\sm \{u\}$ we have 
			\[
			|N_U(w)|\le |U|-|N_U(v)|-1=|U|-d-1.
			\]  
			Because $\{u,v',w\}$ is an MIS and $G$ is tripartite, we must have $N_U(v')\cup N_U(w)=U\sm \{u\}$, which forces 
			\[
			|N_U(w)|\ge |U|-|N_U(v')|-1\ge |U|-d-1,
			\] 
			where we used the maximality of $d$.  Thus, any transversal MIS which contains $u$ but not $v$ must contain some $w\in W$ with $|N_U(w)|=|U|-d-1$.  We also have \[d\ge |N_U(v')|\ge |U|-|N_U(w)|-1=d,\]
			so $|N_U(v')|=d$.
			
			We conclude that there exists a $d$ such that any MIS $(u,v',w')\in U\times V\times W$ must satisfy $|N_U(v')|=d$ (the above analysis deals with $v'\ne v$, and the case $v'=v$ is automatic).  A symmetric argument on $W$ yields a $d'$ for which $|N_U(w')|=d'$ for all $w'$ involved.  Also, we showed above that there exists at least one $w'$ with $|N_U(w')|=|U|-d-1$, so $d' = |U|-d-1$ as desired.
		\end{proof}
		
		\begin{claim}\label{cl:distinguishing-sets}
			For every $u\in U$ in at least one transversal MIS, there exist unique sets $S_u,T_u\sub U$ such that any MIS $(u,v,w)\in U\times V\times W$ has $N_U(v)=S_u$ and $N_U(w)=T_u$.  If $u$ is in at least two transversal MIS's then $S_u,T_u$ partition $U\sm \{u\}$.
		\end{claim}
		\begin{proof}
			The result is trivial if $u$ is in at most one transversal MIS, so we can assume that it is in two distinct MIS's $(u,v_1,w_1), (u,v_2,w_2) \in U \times V \times W$.  Let $S_i=N_U(v_i)$ and $T_i=N_U(w_i)$ for $i=1,2$.  Observe that $S_i\cup T_i=U\sm \{u\}$, as otherwise any $u'\in U\sm (S_i\cup T_i\cup \{u\})$ is non-adjacent to all of $\{u,v_i,w_i\}$, contradicting our assumption that this is an MIS.  By the previous claim we have $|S_i|+|T_i|=|U\sm \{u\}|$, so $S_i,T_i$ must partition $U\sm \{u\}$.
			
			If $S_1\ne S_2$, then because $S_2\cup T_2=U\sm \{u\}$ and $|S_1|=|S_2|$, we must have $S_1\cap T_2 \ne \emptyset$, and $G$ being triangle-free implies that $v_1\not \sim w_2$.  This implies that $\{u,v_1,w_1,w_2\}$ is an independent set, contradicting the maximality of $\{u,v_1,w_1\}$.  Thus we must have $S_1=S_2$ and hence $T_1=T_2$. These hold for any choice of two distinct MIS's $(u, v_1, w_1), (u, v_2, w_2) \in U \times V \times W$, proving the claim.
		\end{proof}
		Let $V_d\sub V$ be the vertices satisfying $|N_U(v)|=d$, let $W_d\sub W$ be those satisfying $|N_U(w)|=|U|-d-1$, and let $U_d$ be the vertices which are in at least one MIS and such that every transversal MIS $(u,v,w)\in U\times V \times W$ satisfies $v\in V_d$ and $w\in W_d$.
		\begin{claim}
			For any $d\ge 0$, we have $\Mt(G[U_d \cup V \cup W],3) \le |V_d|+|W_d|$.
		\end{claim}
		\begin{proof}
			For sets $S,T \subseteq U$ define $V(S)=\{v\in V:N_U(v)=S\}$ and similarly define $W(T)$. If $(u,v,w)\in U_d\times V \times W$ is a transversal MIS in the tripartite graph $G$, then Claim~\ref{cl:distinguishing-sets} implies that $\{v,w\}$ is a transversal MIS in the induced bipartite graph $G[V(S_u)\cup W(T_u)]$. Since each vertex in a bipartite graph can be in at most one transversal MIS, the number of transversal MIS's in a bipartite graph is at most the size of the smaller part, so we conclude that \begin{equation}T(G[U_d\cup V\cup W],3)\le \sum_{u\in U_d} \min \{|V(S_u)|,|W(T_u)|\}.\label{eq:T}\end{equation}
			
			We show that the above quantity is small by constructing an auxiliary graph $G_d$ as follows. Let $V(G_d)=2^{U_1}\cup 2^{U_2}$, where $U_1,U_2$ are disjoint copies of $U$.  For each $u\in U_d$, add an edge from $S_u\in 2^{U_1}$ to $T_u\in 2^{U_2}$ and label it by $u$. In particular, all edge labels in $G_d$ are unique. By Claim~\ref{cl:distinguishing-sets}, the existence of an edge labelled by $u$ implies $S_u, T_u$ form a partition of $U\sm \{u\}$.
			
			We claim that $G_d$ is a forest, i.e. it contains no cycle $S_1,T_1,\ldots,S_k,T_k$ for any $k\ge 2$.  Indeed, assume there were such a cycle whose edges were labeled by $u_1,u'_1,\ldots,u_k,u_k'$. We find that $S_1 \sqcup T_1 = U \sm \{ u_1 \}$, $T_1 \sqcup S_2 = U \sm\{u'_1\}$, $S_2 \sqcup T_2 = U \sm \{u_2\}$, and so on. It is straightforward to prove by induction that this implies $S_\ell=S_1\cup \{u_1,\ldots,u_{\ell-1}\}\sm \{u_1',\ldots,u_{\ell-1}'\}$ for $\ell \ge 1$. One can further show that this claim also applies to $S_{k+1}\coloneqq S_1$, i.e. $S_1 = S_1 \cup \{u_1,\ldots, u_k\} \sm \{u'_1,\ldots, u'_{k}\}$. This is impossible because all the edge labels are distinct vertices of $U$, proving that $G_d$ is a forest.
			
			Because $G_d$ is a forest, one can orient the edges of $G_d$ so that each vertex has in-degree at most 1 (this can be done, for example, by rooting each component and then orienting each edge away from the root).
			For each term in \eqref{eq:T}, we use $\min\{|V(S_u)|,|W(T_u)|\}\le |V(S_u)|$ if the edge labeled $u$ in $G_d$ points to $S_u$, and otherwise we use $|W(T_u)|$ as an upper bound.  Because each vertex of $G_d$ has in-degree at most 1, we have that
			\[
			\Mt(G[U_d\cup V \cup W], 3)\le \sum_{u\in U_d} \min\{|V(S_u)|,|W(T_u)| \le \sum_{|S| =d} |V(S)|+\sum_{|T|=|U|-d-1} |W(T)|=|V_d|+|W_d|,
			\]
			where we used the fact that, as $S$ ranges through all $d$-subsets of $U$ and $T$ through $(|U|-d-1)$-subsets, the $V(S),W(T)$ sets partition $V_d,W_d$ by definition. This completes the proof of the claim.
		\end{proof}
		
		To conclude the lemma, we observe by Claim~\ref{cl:d-exist} that every transversal MIS must contain a vertex in some $U_d$ or a vertex in $U\sm \bigcup U_d$ which lies in exactly one transversal MIS. Thus, by the previous claim,
		\[
		\Mt(G, 3) \le |U\sm \bigcup U_d|+\sum_d (|V_d|+ |W_d|)\le |U|+|V|+|W|,
		\]
		since the sets $V_d,W_d$ are disjoint by construction. This proves the result.
	\end{proof}
    
    This lemma will suffice to prove Theorem~\ref{thm:triangle-free}.
	
	\begin{proof}[Proof of Theorem~\ref{thm:triangle-free}]
		The $k=2$ case was proved in Proposition~\ref{prop:bipartite-case}. The lower bounds follow from Theorem~\ref{thm:constructions}, so it remains to show that $m_3(n,3) = O(n)$ and $m_3(n,4) = O(n^2)$. 
		
		Observe that for any triangle-free graph $G$, by Lemma~\ref{lem:transversal} we have $m(G, 3) \le 12^3 \Mt(G', 3)$ for some induced tripartite subgraph $G'$. By Lemma~\ref{lem:tripartite}, $\Mt(G',3) \le |V(G')| \le |V(G)|$. This proves that $m_3(n,3) = O(n)$.
		
		We claim that the $k=4$ case follows from the $k=3$ case.  Indeed, each $v\in V(G)$ lies in at most $m_3 (n-d(v)-1, 3) = O(n)$ MIS's of size $4$, so \[m_3(n,4) \le O(n^2).\qedhere\] 
	\end{proof}

	\section{A Partial Result for $t=3,k=5$}\label{sec:partial}
For $k\ge 5$, we suspect that $m_3(n,k)=\Theta(n^{k/2})$, i.e. that the lower bound of Theorem~\ref{thm:constructions} is tight.  By Lemma~\ref{lem:transversal}, it suffices to prove this bound for $k$-partite $n$-vertex triangle-free graphs.  Unfortunately we are not quite able to do this for $k=5$, though we can prove the following partial result in this direction.
	\begin{prop}\label{prop:k5}
		Let $G$ be an $n$-vertex 5-partite graph on $V_1\cup \cdots \cup V_5$ which is triangle-free.  If there are no edges in any of the sets $V_1\cup V_3,\ V_2\cup V_4$, or $V_2\cup V_5$, then $G$ contains at most $O(n^{5/2})$ MIS's of size 5.
	\end{prop}
	In particular, this result shows that if all the edges of $G$ lie in the blowup of a 5-cycle (or more generally, a 5-cycle together with two extra edges), then $G$ does not contain significantly more MIS's of size 5 than the construction given in Theorem~\ref{thm:constructions}.
\begin{proof}
By Lemma~\ref{lem:transversal}, it suffices to show that such a graph $G$ has at most $O(n^{5/2})$ transversal MIS's.  Throughout this proof we refer to transversal 5-MIS's simply as MIS's for ease of presentation.

For $a\in V_1,b\in V_2,c\in V_3$, let $D_{a,b,c}$ denote the set of $d\in V_4$ such that there exists an MIS $I$ with $\{a,b,c,d\}\sub I$.  Similarly define $E_{a,b,c}\sub V_5$.  The key claim we need is the following.

\begin{claim}\label{cl:k5}
		For any $b\in V_2$, the sets $Q_{a,b,c}:=D_{a,b,c}\times E_{a,b,c}$ are disjoint for distinct choices of $(a,c)\in V_1\times V_3$.
	\end{claim}
	\begin{proof}
		Assume for contradiction that there exist $(a',c'),(a'',c'')\in V_1\times V_3$ with $(a',c')\ne (a'',c'')$ such that there exist $d\in D_{a',b,c'}\cap D_{a'',b,c''}$ and $e\in E_{a',b,c'}\cap E_{a'',b,c''}$.  Without loss of generality we can assume $c'\ne c''$. By definition of $d$ and $e$, there exist $e',e''\in V_5$ and $d',d''\in V_4$ such that
		\[I_d':=\{a',b,c',d,e'\},\ I''_d:=\{a'',b,c'',d,e''\},\ I'_e:=\{a',b,c',d',e\},I''_e=\{a'',b,c'',d'',e\}\] are MIS's.
		
		We first claim that all of $e,e',e''$ are distinct vertices. Indeed if, say, $e'\in \{e,e''\}$, then $I_d'\cup \{c''\}$ would be an independent set (we have $c''\not\sim a'$ since there are no edges in $V_1\cup V_3$, and $c''\not\sim b,d,e'$ since $c''$  lies in an MIS with each of these if $e'\in \{e,e''\}$), contradicting the maximality of $I_d'$ since $c''\ne c'$. Thus $e'\ne e,e''$, and by a symmetric argument $e''\ne e,e'$.  
		
		We claim that $d\sim e,\ d'\sim e',$ and $d''\sim e''$.   Indeed if, say, $d'\not\sim e'$, then $I_e'\cup \{e'\}$ would also be an independent set (since $e'\not\sim a',c'$), contradicting the maximality of $I'_e$ since $e'\ne e$. A symmetric argument shows $d''\sim e''$, and if $d\not\sim e$, then $I_d'\cup \{e\}$ would be an independent set.  
		
		Define \[\tilde{V}_3=V_3\sm N(b),\hspace{2em} N_x=N(x)\cap \tilde{V}_3\hspace{1em} \forall x\in V(G).\] Note that $c',c''\notin \tilde{V}_3$ since they are in an MIS with $b$.  Because $I_d',I_e'$ are MIS's (and because there are no edges in $V_1\cup V_3$), we must have \begin{equation}N_d\cup N_{e'}=\tilde{V}_3\sm \{c'\}=N_{d'}\cup N_e.\label{eq:N}\end{equation}
		That is, every vertex in $V_3\sm \{c'\}$ which is not adjacent to $b$ must be adjacent to either $d$ or $e'$.  However, because $d\sim e$ and $N_d\sub N(d),N_e\sub N(e)$, we must have $N_d\cap N_e=\emptyset$ since $G$ is triangle-free.  This together with \eqref{eq:N} implies that $N_d\sub N_{d'}$ and $N_e\sub N_{e'}$.  A symmetric argument using $d'\sim e'$ implies that $N_{d'}\sub N_d$ and $N_{e'}\sub N_e$, so we have $N_d=N_{d'},\ N_e=N_{e'}$.  A symmetric argument gives $N_d=N_{d''}$ and $N_{e}=N_{e''}$, but this in total implies \[\tilde{V}_3 \sm \{c'\}=N_d\cup N_{e'}=N_d\cup N_e=N_d\cup N_{e''}=\tilde{V}_3\sm \{c''\},\]
		which is a contradiction to $c'\ne c''$.  This proves the claim.
	\end{proof}
For $b\in V_2$, define $t(b)$ to be the number of MIS's containing $b$, and for $a\in V_1,c\in V_3$ define $t(a,b,c)$ to be the number of MIS's $I$ with $\{a,b,c\}\sub I$.  Observe that $|D_{a,b,c}|=|E_{a,b,c}|=t(a,b,c)$ since every MIS $I$ counted by $t(a,b,c)$ is uniquely determined by specifying either $I\cap V_4$ or $I\cap V_5$, so in particular $|Q_{a,b,c}|=t(a,b,c)^2$.  If $P_b\sub V_1\times V_3$ denotes the set of pairs $(a,c)$ such that $t(a,b,c)\ge 1$, then by the claim we have for all $b\in V_2$ that
\[
\sum_{(a,c)\in P_{b}}t(a,b,c)^{2}=\sum_{a,c\in P_{b}}|Q_{a,b,c}|\le |V_4\times V_5|\le n^{2}.
\]
It is not difficult to see that $\sum_{(a,c)\in P_{b}}t(a,b,c)=t(b)$. Using this with the above inequality and the Cauchy-Schwarz
inequality, we find for all $b\in V_2$ that
\begin{equation}
t(b)^{2}=\l(\sum_{a,c\in P_{b}}t(a,b,c)\r)^{2}\le|P_{b}|\sum_{a,c\in P_{b}}t(a,b,c)^{2}\le n^{2}|P_{b}|.
\label{eq:tbornottb}\end{equation}
Observe that if $I=\{a,b,c,d,e\}$ is an MIS with $a\in V_1,b\in V_2,c\in V_3$, then we must have $N(a)\cup N(c)=V_2\sm \{b\}$ since $I$ is maximal and since there are no edges from $V_2$ to $V_4\cup V_5$.  Thus a pair $(a,c)\in V_1\times V_3$ can be in at most one $P_b$ set (namely the $b$ satisfying $N(a)\cup N(c)=V_2\sm \{b\}$), so the sets $P_{b}$ are distinct subsets
of $V_1\times V_3$ for distinct choices of $b$.  This implies
\[
\sum_{b\in V_2}|P_{b}|\le n^{2}.
\]

Let $\c{I}$ denote the set of MIS's of $G$.  Then $\sum_{b\in V_2}t(b)=|\c{I}|$, so by applying Cauchy-Schwarz, \eqref{eq:tbornottb}, and the above inequality, we find
\[
|\c{I}|^{2}=\Big(\sum_{b}t(b)\Big)^{2}\le n\sum_{b\in V_2}t(b)^{2}\le n^{3}\sum_{b\in V_2}|P_{b}|\le n^{5},
\]
as desired.
\end{proof}

\section{Hypergraphs}\label{sec:hypergraphs}
Lastly, we consider an analogous problem for hypergraphs.  We define $m_t^r(n,k)$ to be the maximum number of MIS's of size $k$ in a $K_t^r$-free $r$-uniform hypergraph on $n$ vertices.  Somewhat surprisingly, we are able to completely determine the order of magnitude of $m_t^r(n,k)$ for $k$ fixed and $r\ge 3$.
	\begin{prop}
		For $n\ge 4$, we have
		\[m_4^3(n,2)=n-1.\]
		For any other set of parameters satisfying $r\ge 3,\ k\ge r-1,$ and $t\ge r+1$; if $s\in \{0,1,\ldots,k-1\}$ with $n\equiv s\mod k$ then
		\[m_t^r(n,k)\ge \floor{n/k}^{k-s}\ceil{n/k}^s.\]
	\end{prop}

	Note that determining $m_t^r(n,k)$ is trivial if $k<r-1$ (since any set of $k+1$ vertices is an independent set) or if $t= r$.  Also note that we always have the upper bound $m_t^r(n,k)\le {n\choose k}\le (en/k)^k$, so these lower bounds are close to best possible.

	\begin{proof}
		We first consider the case $r=3,\ t=4,\ n\ge 4,$ and $k=2$.  A lower bound comes from taking $H$ to be the 3-graph consisting of every edge containing a given vertex $v$, since in this case every set $\{u,v\}$ is a 2-MIS.  For the upper bound, observe that if $\{u,v\}$ is an MIS of a hypergraph $H$, then we must have that $\{u,v,w\}\in E(H)$ for all $w\ne u,v$ (otherwise the set would not be maximal).  Thus if $\{u,v\},\{w,x\}$ are disjoint MIS's then $\{u,v,w,x\}$ form a $K_4^3$ in $H$, which we have assumed not to be the case.  We conclude that every MIS of size 2 pairwise intersects, and by the Erd\H{o}s-Ko-Rado Theorem~\cite{erdos1961intersection}, there can be at most $n-1$ such pairs.
		
		We now assume our parameters are not of this form and construct an $r$-graph $H$ as follows.  Partition $V(H)$ into sets $V_1,\ldots,V_k$ with $s$ of these having size $\ceil{n/k}$ and the rest size $\floor{n/k}$.  The edges of $H$ will be the $r$-sets which intersect some $V_i$ in exactly two vertices and each $V_{i+j}$ with $1\le j\le r-2$ in exactly one vertex.  Here we write our indices mod $k$, and this construction is well defined since $k\ge r-1$.
		
		We claim that any set $I=\{v_1,\ldots,v_k\}$ with $v_i\in V_i$ is an MIS.  Indeed, $I$ is an independent set because $I$ intersects each $V_i$ in exactly one vertex, so it contains none of the edges of $H$.  Moreover, any additional vertex will cause $I$ to intersect some $V_i$ in exactly two vertices and each $V_{i+j}$ in exactly one vertex, so this set is also maximal.  There are exactly $\floor{n/k}^{k-s}\ceil{n/k}^s$ ways to choose such an MIS $I$, so we will be done provided we can show that $H$ is $K_t^r$-free for any $t\ge r+1$.
		
		Assume for contradiction that $H$ contains a set $K$ which induces a $K_{r+1}^r$.  Observe that if $S\sub K$ has at most $r$ vertices, then there must exist some edge containing $S$ in $H$.  Because every edge of $H$ contains at most 2 vertices of any $V_i$ set, for $r\ge 3$ we must have $|V_i\cap K|\le 2$ for all $i$.  Similarly because each edge of $H$ contains at most 3 vertices of any $V_i\cup V_j$ set, for $r\ge 4$ we must have that there exists no $i\ne j$ with $|V_i\cap K|=|V_j\cap K|=2$, but this implies that $K$ can contain at most 1 edge, a contradiction.  Thus we can assume $r=3$ and that $|V_i\cap K|,|V_j\cap K|=2$ for some $i\ne j$.  Without loss of generality we can assume that $|V_1\cap K|=2$ (if no such $i$ exists then $K$ contain no edges).  The only vertices which are in an edge containing two vertices of $V_1$ are those in $V_2$, so the other vertices of $K$ must all be in $V_2$, and hence $|K|=4$ since $|V_j\cap K|=2$ for some $j\ne 1$.  Such a set of vertices is a clique in $H$ only if $k=2$, which we assumed not to be the case.
	\end{proof}
	
	\section{Concluding Remarks}\label{sec:concluding}
	Our main open problem is the following, which claims that the lower bounds from Theorem~\ref{thm:constructions} are essentially tight when $k$ is sufficiently large in terms of $t$.   
	
	\begin{conj}\label{conj:main}
	For all $t\ge 3$, $k\ge 2(t-1)$,
	\[
	m_t(n,k) =O(n^{(t-2)k/(t-1)}).
	\]
	\end{conj}
	
	This conjecture seems difficult, even when restricted to triangle-free graphs where we conjecture $m_3(n,k)=O(n^{k/2})$ for all $k\ge 4$.  In particular, we do not even know if the following is true.
	\begin{conj}
	There exists an integer $k>4$ such that
	\[m_3(n,k)=O(n^{k-3}).\]
	\end{conj}
	We note that it is easy to prove $m_3(n,k)=O(n^{k-2})$ for all $k\ge 4$ by using  $m_3(n,4)=O(n^2)$ from Theorem~\ref{thm:triangle-free}.
	
	As a consequence of Proposition~\ref{prop:k5}, any $G$ which is the subgraph of a blowup of $C_5$ has at most $O(n^{5/2})$ 5-MIS's.  One might be able to extend these ideas to work when $G$ is the subgraph of a blowup of any graph $K$ which is triangle-free.
	\begin{conj}
	Let $G$ be an $n$-vertex $k$-partite graph on $V_1\cup \cdots \cup V_k$.  Assume there exist no indices $a,b,c$ such that there exist edges in $V_{a}\cup V_b,\ V_b\cup V_c$, and in $V_a\cup V_c$.  Then $G$ contains at most $O(n^{k/2})$ $k$-MIS's.
	\end{conj}
	
	Turning our attention to larger $t$, we ask the following question.
	\begin{quest}
	Is it true that for all $t\ge 3$,
	\[m_t(n,t)=O(n^{t-2}).\]
	\end{quest}
	If this is true, then this together with Proposition~\ref{prop:smallK} can be used to show that the lower bounds of Theorem~\ref{thm:constructions}(a) are tight up to the $o(1)$ term for all $k<2(t-1)$.  We note that Theorem~\ref{thm:triangle-free} implies this is true for $t=3$, and it is possible that similar ideas could be used to handle larger values of $t$.
	
	Lastly, it is natural to ask if the $o(1)$ terms in the bounds of Theorem~\ref{thm:constructions} are necessary.  The simplest case of this question is the following.
	\begin{quest}\label{quest:K4}
	Is it true that
	\[m_4(n,3)=n^{2-o(1)}.\]
	\end{quest}
	The following partial result is easy to show.
	\begin{prop}
	    If $G$ is an $n$-vertex tripartite graph, then $G$ has at most $n^{2-o(1)}$ 3-MIS's.
	\end{prop}
	\begin{proof}[Sketch of Proof]
	It is not difficult to show that $G$ has at most $O(n)$ 3-MIS's which are not transversal, so it suffices to upper bound the number of transversal MIS's.  Let $G'$ be the graph on the same vertex set of $G$ where an edge is added between two vertices $u,v$ if and only if they are in different parts and if $uv\notin E(G)$.  It is not difficult to see that if $\{u,v,w\}$ is a transversal MIS of $G$, then $uvw$ is a triangle in $G'$, and moreover it is the unique triangle containing each of $uv,vw,uw$.  Let $G''\sub G'$ be the subgraph consisting of every edge which is in a (unique) triangle.  It is a well known consequence of the regularity lemma that such a graph $G''$ has at most $n^{2-o(1)}$ edges, i.e. $G$ has at most $n^{2-o(1)}$ pairs which lie in a trasnversal MIS.  From this the result follows.
	\end{proof}
	Using similar ideas it may be possible to positively answer Question~\ref{quest:K4} for $K_4$-free graphs in general.  We emphasize that in this setting Lemma~\ref{lem:transversal} does not apply since our graph is not triangle-free, so it is not enough to simply solve the problem in the tripartite case.
	
	\bibliographystyle{abbrv}
	\bibliography{MIS}
\end{document}